\documentclass[11pt]{amsart}
\usepackage[english]{babel}
\usepackage{amssymb,amscd}

\usepackage[all]{xy}

\usepackage[latin1]{inputenc}

\usepackage{euscript}
\usepackage{amsmath}
\usepackage{ifpdf}

\usepackage{amsfonts}
\usepackage{mathrsfs}

\usepackage{amsmath}
\usepackage{mathpazo}
\usepackage{hyperref}

\newtheorem{theorem}{Theorem}[section]
\newtheorem{lemma}[theorem]{Lemma}
\newtheorem{corollary}[theorem]{Corollary}

\newtheorem{remark}[theorem]{Remark}
\newtheorem{example}[theorem]{Example}
\newtheorem{examples}[theorem]{Examples}
\newtheorem{definition}[theorem]{Definition}

\newtheorem{question}[theorem]{Question}
\newtheorem{convention}[theorem]{Convention}

\def\sqr#1#2{{\vcenter{\hrule height.#2pt
\hbox{\vrule width.#2pt height#1pt \kern#1pt \vrule width.#2pt}
\hrule height.#2pt}}}

\def\qed{\hspace*{\fill} $\square$}

\begin{document}


\title[Tensor products and homological conjectures for Ulrich modules]{Tensor products and solutions to two homological conjectures for Ulrich modules}

\author{Cleto B.~Miranda-Neto}
\address{Departamento de Matem\'atica, Universidade Federal da
Para\'iba, 58051-900, Jo\~ao Pessoa, PB, Brazil.}
\email{cleto@mat.ufpb.br}

\author{Thyago S.~Souza}
\address{Unidade Acad\^emica de Matem\'atica, Universidade Federal de
Campina Grande, 58429-970, Campina Grande, PB, Brazil.}
\email{thyago@mat.ufcg.edu.br}




\begin{abstract}
We address the problem of when the tensor
product of two finitely generated modules over a Cohen-Macaulay local ring is Ulrich in the generalized sense of Goto et al., and in particular in the original sense from the 80's. As applications, besides freeness criteria for modules, characterizations of complete intersections, and an Ulrich-based approach to the long-standing Berger's conjecture, we show that two celebrated homological conjectures, namely the Auslander-Reiten and the Huneke-Wiegand problems, are true for the class of Ulrich modules. 
\end{abstract}

\maketitle

\section{Introduction}

The theory of Ulrich modules -- both the classical one initiated by B. Ulrich 
 \cite{ulrich84} (see also  \cite{BHU}, \cite{H-K}, \cite{ulrich91}) and its generalization proposed by Goto et al. \cite{goto2014} -- is nowadays recognized as being an active area of research in commutaive algebra as well as in algebraic geometry (since the inception of the notions of Ulrich sheaves and bundles). For some history and justifications for the interest in the theory, we refer to \cite{kobayashi2018}, \cite{MirandaQueirozSouza} and their suggested references on the subject.

Let $R$ be a Cohen-Macaulay local ring and let $M, N$ be finitely generated $R$-modules. The present note addresses the problem as to when the tensor product $M\otimes_RN$ is Ulrich in the generalized sense of Goto et al., and in particular in the classical context from the 80's, i.e., its minimal number of generators equals its multiplicity and, in addition, maximal Cohen-Macaulayness takes place. Significant efforts have been made to study tensor products with the latter property (see, e.g., \cite{HIW, HW, KYT2022, lyle.montano}), so the present paper goes one step further, and additionally provides a number of applications.

Conventions and preliminaries  are given in
Section \ref{conv+basics}, while Section \ref{prep} invokes some preparatory lemmas 
which are used throughout the paper. In Section \ref{aux} we establish our main general results on the Ulrichness of tensor products over Cohen-Macaulay local rings, namely, theorems \ref{teo.ulrich.tensor.products}, \ref{teo3.ulrich.tensor.products}, and \ref{teo2.ulrich.tensor.products}, where the first two theorems require the vanishing of some (finitely many) Ext modules, and the last one deals with the case where $M$ has finite Gorenstein dimension and $N=\omega_R$ (a canonical module of $R$, in case one exists). 

Finally, in Section \ref{appli} we describe applications  which we divide into four subsections. In \ref{appli1} we study multiplicity of tensor products and derive freeness criteria for finite modules. In \ref{appli2} we show that Urich modules can be used to characterize codimension 3 complete intersection ideals. In \ref{appli3} we develop an Ulrich-based approach to the long-standing Berger's conjecture on modules of differentials. We close the paper with \ref{appli4}, where we show that the celebrated Huneke-Wiegand and Auslander-Reiten conjectures (for the latter we assume that the given local ring is a Cohen-Macaulay domain) are true for the class of Ulrich modules; to this end, we prove a central result  (Corollary \ref{cor5.ulrich.tensor.products}) which in fact gives a characterization of regular local rings.

\section{Conventions, and basics on Ulrich ideals and modules}\label{conv+basics}

First we make some conventions that will be in force throughout the entire paper. By a {\it ring} we mean a commutative Noetherian ring with $1\neq 0$, and if $R$ is a ring then by a {\it finite} $R$-module we mean a finitely generated $R$-module. Whenever $R$ is a local ring, ${\mathfrak m}$ stands for its maximal ideal. Moreover, if $R$ possesses a canonical module (e.g., if $R$ is a complete Cohen-Macaulay local ring), then this module is unique up to isomorphism and is denoted $\omega_R$.

Let $R$ be a local ring, $M$ a finite $R$-module, and $I\neq R$ an ideal of definition of $M$, i.e., ${{\mathfrak m}}^nM\subset IM$ for some $n\geq 1$.  We denote by $\nu(M)$ and
$\textrm{e}_I^0 (M)$, respectively, the minimal number of generators of $M$ and the multiplicity of $M$ with respect to $I$. If $I
= {\mathfrak m}$, we  write $\textrm{e}(M)$ instead of $\textrm{e}_{{\mathfrak m}}^0(M)$. Recall that $M$ is said to have a (generic, constant) rank, which is typically denoted ${\rm rk}_RM$, if  $M\otimes_RQ$ is a free $Q$-module of rank ${\rm rk}_RM$, where $Q$ is the total fraction ring of $R$. For instance, any ideal of $R$ containing a regular element has rank 1 as an $R$-module, and if $R$ is a domain then ${\rm rk}_RM$ always exists. In case $M$ has a rank, we have the useful formula $\textrm{e}(M)={\rm rk}_RM\, {\rm e}(R)$. 

Next, recall that $M$ is a maximal Cohen-Macaulay $R$-module if ${\rm depth}_RM={\rm dim}\,R$, where by depth we always mean ${\mathfrak m}$-depth. Note the zero module is not maximal Cohen-Macaulay since, by a widely accepted convention, its depth is set to be $+\infty$. Now if the local ring $R$ is Cohen-Macaulay and $M$ is maximal Cohen-Macaulay, then there is an inequality $\nu(M) \leq \textrm{e}(M)$ (see \cite[Proposition 1.1]{BHU}).

\begin{definition}\label{def.ulrich.module} \rm Let $R$ be a local ring. A finite $R$-module $M$ is \emph{Ulrich} if $M$ is maximal Cohen-Macaulay and $\nu(M)=\textrm{e}(M)$.
\end{definition}

The study of the above property initiated in the 80's; see \cite{ulrich84} (the term {\it Ulrich module} was coined in \cite{H-K}). Several classes of examples of Ulrich modules can be found in the literature; for example, if $R$ is a 1-dimensional Cohen-Macaulay local ring with multiplicity ${e}$, then ${{\mathfrak m}}^{{e}-1}$ is Ulrich.


\begin{convention}\label{convention}\rm Throughout the entire paper, whenever $R$ is a Cohen-Macaulay local ring, we let  $I$ stand for an ${\mathfrak m}$-primary ideal containing a parameter ideal $Q$ as a reduction. It is well-known, for example, that any ${\mathfrak m}$-primary ideal of $R$ has this property if the residue class field $R/{\mathfrak m}$ is infinite, or if $R$ is 1-dimensional and analytically irreducible.
\end{convention}







\begin{definition}\label{def.urich.ideal}\rm (\cite{goto2014})
Let $R$ be a Cohen-Macaulay local ring. 
The ideal $I$
is \emph{Ulrich} if $I^2 =
{Q}I$ and $I/I^2$ is a free $R/I$-module.
\end{definition}

\begin{remark}\label{obs1.def.urich.ideal}\rm Recall that if $R$ is a Cohen-Macaulay local ring, then the ideal $I$ is \emph{Gorenstein} if the quotient ring $R/I$ is Gorenstein. Now, if $R$ is Gorenstein then according to \cite[Corollary 2.6]{goto2014} every Ulrich ideal is Gorenstein.
\end{remark}


\begin{definition}\label{def.I.urich.module}\rm (\cite{goto2014})
Let $R$ be a Cohen-Macaulay local ring. A finite
$R$-module $M$ is \emph{Ulrich
with respect to} $I$ if the following conditions hold:
\begin{itemize}
             \item[(i)] $M$ is a maximal Cohen-Macaulay $R$-module;
             \item[(ii)] $IM = {Q}M$;
             \item[(iii)] $M/IM$ is a free $R/I$-module.
\end{itemize}
\end{definition}

\begin{examples}\rm \label{exemplo} (i) (\cite[Proposition 3.9]{Ku}) Let $S=K[\![x, y, z]\!]$ be a formal power series ring over an infinite field $K$, and fix any regular sequence $\{f, g, h\}\subset \mathfrak{m}$. Consider the 1-dimensional Cohen-Macaulay local ring $R=S/(f^2-gh, g^2-hf, h^2-fg)$. Then, $(f, g, h)R$ is an Ulrich ideal. 
\medskip

(ii) Clearly, if ${\rm dim}\,R=1$ and the ideal $I$ is Ulrich, then $I$ is an Ulrich $R$-module with respect to $I$.

\medskip

(iii) There is an interesting recipe to obtain Ulrich modules from Ulrich ideals (in the setting of Convention \ref{convention}). Indeed, by \cite[Theorem 4.1]{goto2014}, if $I$ is Ulrich but not a parameter ideal then, for any $i\geq {\rm dim}\,R$, the $i$-th syzygy module of $R/I$ is an Ulrich $R$-module with respect to $I$. 

\end{examples}

\begin{remark}\label{obs.I.urich.module}\rm Let us point out that, according to the discussion in \cite[paragraph after Definition 1.2]{goto2014},
condition (ii) of Definition \ref{def.I.urich.module} is equivalent to saying that $\textrm{e}_{I}^0 (M) =
\ell_R(M/IM)$. In particular, if $I = {\mathfrak m}$ then (ii) is the same as  $\textrm{e}(M) = \nu(M)$. This shows that $M$ is an Ulrich module with respect to ${\mathfrak m}$ if and only if $M$ is an Ulrich module in the sense of Definition
\ref{def.ulrich.module}.

\end{remark}


\section{Preparatory results}\label{prep}

Throughout this auxiliary section, 
$R$ stands for a Cohen-Macaulay local ring of dimension $d\geq 1$, possessing a canonical module $\omega_R$ (the existence of $\omega_R$ is not required in Lemma \ref{main-lemma}), and $M, N$ denote finite $R$-modules. 

Following standard notations, we write $M^*={\rm Hom}_R(M, R)$ and $M^{\dagger}={\rm Hom}_R(M, \omega_R)$. In particular, $M^*\cong M^{\dagger}$ if $R$ is Gorenstein.
The ({\it Auslander}) \emph{transpose} $\textrm{Tr}\,M$ of 
$M$ is defined as the cokernel of the map $\partial_1^*
= \textrm{Hom}_{R} (\partial_1, R)$, i.e., the dual of the first differential map
$\partial_1$ in a minimal free resolution of $M$ over $R$. Hence, a minimal free presentation $F_1 \stackrel{\partial_1}{\rightarrow} F_0 \rightarrow
M \rightarrow 0$ induces
an exact sequence $$0 \longrightarrow M^{*} \longrightarrow
F_0^{*} \stackrel{\partial_1^{*}}{\longrightarrow} F_1^{*} \longrightarrow
{\rm Tr}\,M \longrightarrow 0.$$ Note ${\rm Tr}\,M$ is uniquely determined up to isomorphism.
Next we collect the auxiliary facts that will be used in the proof of our main results.

\begin{lemma}\label{L-M}{\rm (\cite[Lemma 3.4(1)]{lyle.montano})} Suppose $N$ is maximal Cohen-Macaulay. If ${\rm Ext}_R^i(M,N) = 0$ for all $i=1, \ldots d$, then $M\otimes_RN^{\dagger}$ is  maximal Cohen-Macaulay.
\end{lemma}

\begin{lemma}\label{Taka}{\rm (\cite[Theorem 2.2]{KYT2022})} Suppose $N$ is maximal Cohen-Macaulay. Then, $M\otimes_RN $ is maximal Cohen-Macaulay
if and only if ${\rm Tr}\,M \otimes_RN^{\dagger}$ is maximal Cohen-Macaulay. \end{lemma}

\begin{lemma}\label{main-lemma}{\rm (\cite[Theorem 3.2]{MirandaQueirozSouza})}
Suppose $M, N$
are maximal Cohen-Macaulay $R$-modules such that
$\textrm{\emph{Hom}}_{R} (M, N) \neq 0$ and
$\textrm{\emph{Ext}}_{R}^i(M, N)=0$ for all $i=1, \ldots, n$, where either $n=d-1$ {\rm (}in this case, the condition is vacuous if $d=1${\rm )} or $n=d$. Let $I$ and $Q$ be as in Convention \ref{convention}. Assume that $M$ {\rm (}resp.\,$N${\rm )} is an Ulrich $R$-module with respect
to $I$, and consider the following conditions:
\begin{itemize}
             \item[(i)] $\textrm{\emph{Hom}}_{R} (M, N)$ is an Ulrich
$R$-module with respect to $I$;
             \item[(ii)] $\textrm{\emph{Hom}}_{R} (M, N) / I \textrm{\emph{Hom}}_{R} (M, N)$ is a free
             $R/I$-module;
             \item[(iii)]  $\textrm{\emph{Hom}}_{R/Q} (R/I, N/QN)$ {\rm (}resp. $\textrm{\emph{Hom}}_{R/Q} (M/QM, R/I)${\rm )} is a free
             $R/I$-module.
\end{itemize}
Then the following statements hold:
\begin{enumerate}
             \item[(a)] If $d=1$ or $n = d-1\geq 1$ then {\rm (i)} $\Leftrightarrow$ {\rm
             (ii)};
             \item[(b)] If $n = d\geq 1$ then {\rm (i)} $\Leftrightarrow$ {\rm
             (ii)} $\Leftrightarrow$ {\rm
             (iii)}.

\end{enumerate}
\end{lemma}


\begin{lemma}\label{ulr}{\rm (\cite[Corollary 3.5]{MirandaQueirozSouza})}
Assume that $M$ is a maximal Cohen-Macaulay $R$-module, and that the ideal $I$ is Gorenstein. The following
assertions are equivalent:
\begin{itemize}
             \item[(i)] $M$ is an Ulrich
$R$-module with respect to $I$;
             \item[(ii)] $M^{\dagger}$ is an Ulrich
$R$-module with respect to $I$.
\end{itemize}
\end{lemma}


\section{Ulrich tensor products}\label{aux}
As before, we let $(R, \mathfrak{m})$ stand for a Cohen-Macaulay local ring of dimension $d\geq 1$ and possessing a canonical module, and the ideals $Q\subset I$ be as in Convention \ref{convention}. 


\subsection{First result and corollaries}  Our first theorem is as follows.

\begin{theorem}\label{teo.ulrich.tensor.products}
Let $M$ be a maximal Cohen-Macaulay $R$-module
and $N$ an Ulrich $R$-module with respect to $I$. Suppose $I$ is Gorenstein and $$\textrm{\emph{Ext}}_{R}^i(M,
N^{\dag})=0 \quad \mbox{for \,all} \quad i=1, \ldots, d.$$ The following
conditions are equivalent:
\begin{itemize}
\item[(i)] $M \otimes_R N$ is an Ulrich
$R$-module with respect to $I$;
             \item[(ii)] $\textrm{\emph{Hom}}_{R} (M, N^{\dag})$ is an Ulrich
$R$-module with respect to $I$;
             \item[(iii)] $\textrm{\emph{Hom}}_{R} (M, N^{\dag}) / I \textrm{\emph{Hom}}_{R} (M, N^{\dag})$ is
             $R/I$-free;
             \item[(iv)] $\textrm{\emph{Hom}}_{R/Q} (M/QM, R/I)$ is
             $R/I$-free.
\end{itemize}
\end{theorem}
\begin{proof} Since $I$ is Gorenstein and $N$ is Ulrich with respect to
$I$, it follows from Lemma \ref{ulr}
that $N^{\dag}$ is Ulrich with respect to $I$. In particular,
$N^{\dag}$ is maximal Cohen-Macaulay. As $\textrm{Ext}_{R}^i(M,
N^{\dag})=0$ for all $i=1, \ldots, d$, Lemma \ref{L-M} implies that the module $M \otimes_R (N^{\dag})^{\dag}
\cong M \otimes_R N$ is maximal Cohen-Macaulay as well. Thus, by Lemma \ref{Taka}, $\textrm{Tr} M \otimes_R N^{\dag}$ has the same property. Now, pick a minimal free presentation $R^{n} \rightarrow
R^{m} \rightarrow M \rightarrow 0$. It is easy to see that it induces exact sequences
\begin{equation}\label{seq1} 0 \longrightarrow (\textrm{Tr} M \otimes_R N^{\dag})^{\dag}
\longrightarrow N^{\oplus n} \longrightarrow
N^{\oplus m} \longrightarrow M \otimes_R N \longrightarrow 0,\end{equation}
\begin{equation}\label{seq2}0
\longrightarrow \textrm{Hom}_{R} (M, N^{\dag}) \longrightarrow
(N^{\dag})^{\oplus m} \longrightarrow (N^{\dag})^{\oplus n}
\longrightarrow \textrm{Tr} M \otimes_R N^{\dag} \longrightarrow 0.\end{equation} Note that, since $\textrm{Tr} M \otimes_R N^{\dag}$ is maximal
Cohen-Macaulay, the sequence (\ref{seq2}) yields, in turn, the exact sequence
\begin{equation}\label{seq3}0 \longrightarrow (\textrm{Tr} M \otimes_R N^{\dag})^{\dag}
\longrightarrow N^{\oplus n} \longrightarrow
N^{\oplus m} \longrightarrow \textrm{Hom}_{R} (M,
N^{\dag})^{\dag} \longrightarrow 0.\end{equation} In particular, from (\ref{seq1}) and (\ref{seq3}), we obtain an isomorphism  
 \begin{equation}\label{iso}M \otimes_R N
\cong \textrm{Hom}_{R} (M, N^{\dag})^{\dag}.\end{equation} It follows from Lemma \ref{ulr} that $M
\otimes_R N$ is Ulrich with respect to $I$ if and only if the module
$\textrm{Hom}_{R} (M, N^{\dag})^{\dag\dag}$ is Ulrich with
respect to $I$. Since $\textrm{Tr} M \otimes_R N^{\dag}$ and $N$
are maximal Cohen-Macaulay, the sequence (\ref{seq2}) forces $\textrm{Hom}_{R} (M, N^{\dag})$ to be maximal
Cohen-Macaulay as well. Consequently, $\textrm{Hom}_{R} (M, N^{\dag}) \cong \textrm{Hom}_{R} (M,
N^{\dag})^{\dag\dag}$ and therefore $M \otimes_R N$ is Ulrich
with respect to $I$ if and only if $\textrm{Hom}_{R} (M,
N^{\dag})$ is Ulrich with respect to $I$.

Finally, because $M$ and $N$ are non-zero, (\ref{iso}) yields $\textrm{Hom}_{R} (M, N^{\dag}) \neq 0$.
Now we apply Lemma \ref{main-lemma}(b).\qed
\end{proof}



\begin{corollary}\label{cor1.ulrich.tensor.products}
Let $M$ be an Ulrich $R$-module with respect to
$I$. Suppose $I$ is Gorenstein and
$\textrm{\emph{Ext}}_{R}^i(M, M)=0$ for all $i=1, \ldots, d$.
Then the following conditions are equivalent:
\begin{itemize}
\item[(i)] $M \otimes_R M^{\dag}$ is an Ulrich
$R$-module with respect to $I$;
             \item[(ii)] $\textrm{\emph{End}}_{R}M$ is an Ulrich
$R$-module with respect to $I$;
             \item[(iii)] ${\rm End}_{R}M/I {\rm End}_{R}M$ is
             $R/I$-free;
             \item[(iv)] ${\rm Hom}_{R/Q} (M/QM, R/I)$ is
             $R/I$-free.
\end{itemize}
\end{corollary}
\begin{proof}
Since $I$ is Gorenstein and $M$ is Ulrich with respect to $I$, Lemma \ref{ulr} yields that
$M^{\dag}$ is Ulrich with respect to $I$. Now, we employ Theorem
\ref{teo.ulrich.tensor.products} with $N = M^{\dag}$. \qed
\end{proof}




\medskip

Applying Theorem \ref{teo.ulrich.tensor.products} with $I =
{\mathfrak m }$ as in the classical context, and taking Remark \ref{obs.I.urich.module} into
account, we immediately obtain the
following consequence which is the Ulrich version of Lemma \ref{L-M}.

\begin{corollary}\label{cor3.ulrich.tensor.products}
Let $M$ be a maximal Cohen-Macaulay $R$-module
and $N$ an Ulrich $R$-module. If $\textrm{\emph{Ext}}_{R}^i(M,
N^{\dag})=0$ for all $i=1, \ldots, d$, then $M\otimes_R N$ is an Ulrich $R$-module.
\end{corollary}


\begin{corollary}\label{cor4.ulrich.tensor.products}
Let $M$ be an Ulrich $R$-module. If
$\textrm{\emph{Ext}}_{R}^i(M, M)=0$ for all $i=1, \ldots, d$,
then $M \otimes_R M^{\dag}$ is an Ulrich
$R$-module.
\end{corollary}

\subsection{More results} In this subsection, we establish further results on the Ulrichness of tensor products.

\begin{theorem}\label{teo3.ulrich.tensor.products}
Let $M, N$ be finite $R$-modules such that
$$\textrm{\emph{Ext}}_{R}^i(\textrm{\emph{Tr}} M, N)=0 \quad \mbox{for \,all} \quad 
i=1, \ldots, d+1.$$ Assume that $I$ is Gorenstein. In addition, suppose
$M/IM$ is a free $R/I$-module, $M^*$ is Ulrich with respect to $I$, and $N$ is maximal Cohen-Macaulay {\rm (}resp. $M^*$ is maximal Cohen-Macaulay and $N$ is Ulrich with respect to $I${\rm )}. Then $M \otimes_R N$ is Ulrich with
respect to $I$ if and only if $N/I N$ {\rm (}resp.\, $M/I
M${\rm )} is a free $R/I$-module.
\end{theorem}
\begin{proof} By \cite[Proposition 12.5]{LeuschkeWiegand2012}, we
have the exact sequence
$$0 \longrightarrow \textrm{Ext}_{R}^1(\textrm{Tr} M, N)
\longrightarrow M \otimes_R N \longrightarrow \textrm{Hom}_{R}(M^*, N)
\longrightarrow \textrm{Ext}_{R}^2(\textrm{Tr} M, N) \longrightarrow 0.$$
As in particular $\textrm{Ext}_{R}^i(\textrm{Tr} M, N)=0$ for $i = 1, 2$, it
follows that $M \otimes_R N \cong \textrm{Hom}_{R}(M^*, N)$.
Now, notice that $$\textrm{Ext}_{R}^i(M^*, N) =
\textrm{Ext}_{R}^i({\rm Syz}_2\textrm{Tr} M, N) =
\textrm{Ext}_{R}^{i+2}(\textrm{Tr} M, N) =0,$$ for all $i=1,
\ldots, d-1$ (in case $d\geq 2$), where ${\rm Syz}_2$ denotes the second-syzygy operator over $R$. First, suppose
$M/IM$ is a free $R/I$-module, $M^*$ is Ulrich with respect to $I$, and $N$ is maximal Cohen-Macaulay. Then by Lemma \ref{main-lemma}(a), the module $M \otimes_R N$ is Ulrich with
respect to $I$ if and only if $R / I \otimes_R (M \otimes_R
N)$ is a free $R/I$-module. Since $R / I \otimes_R M \cong
M/I M \cong (R / I)^m$ for some integer $m > 0$, we get
\begin{eqnarray*}
R / I \otimes_R (M \otimes_R N) &\cong& (R / I \otimes_R M)
\otimes_{R/I} (R / I \otimes_R N) \\ &\cong&  (R / I)^m
\otimes_{R/I} (N/IN)\cong (N/ I N)^{\oplus m}. \end{eqnarray*} Therefore, $R / I
\otimes_R (M \otimes_R N)$ is a free $R/I$-module if and only if
$N/I N$ is a free $R/I$-module. The case where $M^*$ is maximal Cohen-Macaulay and $N$ is Ulrich with respect to $I$ is completely similar. \qed
\end{proof}

\medskip

Next we derive a number of corollaries in the Gorenstein case.

\begin{corollary}\label{corteo3.ulrich.tensor.products}
Suppose $R$ is Gorenstein. Let $M,N$ be
finite $R$-modules with
$\textrm{\emph{Ext}}_{R}^i(\textrm{\emph{Tr}} M, N)=0$ for all
$i=1, \ldots, d+1$. Assume that $I$ is Gorenstein. In addition, suppose
$M$ is Ulrich  with respect
to $I$ and $N$ is maximal Cohen-Macaulay. Then $M \otimes_R N$ is Ulrich with
respect to $I$ if and only if $N/I N$ is a free $R/I$-module.
\end{corollary}
\begin{proof} Since $R$ is Gorenstein, Lemma \ref{ulr} gives that $M^*$ is Ulrich with
respect to $I$. Now we apply the theorem above.\qed
\end{proof}


\begin{remark}\rm Under the conditions of Corollary \ref{corteo3.ulrich.tensor.products}, if in addition we suppose $IN=QN$, then we derive that $M \otimes_R N$ is Ulrich with
respect to $I$ if and only if $N$ is Ulrich with
respect to $I$.
\end{remark}


\begin{corollary}\label{cor1.ulrich.tensor.products3}
Suppose $R$ is Gorenstein. Let $M, N$ be
Ulrich $R$-modules with respect to $I$ such that
$\textrm{\emph{Ext}}_{R}^i(\textrm{\emph{Tr}} M, N)=0$ for all
$i=1, \ldots, d+1$. Assume that $I$ is Gorenstein.
Then $M \otimes_R N$ is Ulrich  with respect to
$I$.
\end{corollary}

\begin{corollary}\label{cor2.ulrich.tensor.products3}
Suppose $R$ is Gorenstein. Let $M,N$ be
maximal Cohen-Macaulay $R$-modules such that
$\textrm{\emph{Ext}}_{R}^i(\textrm{\emph{Tr}} M, N)=0$ for all
$i=1, \ldots, d+1$. Assume that $M \otimes_R N$ is  Ulrich
with respect to $Q$. Then $M$ is Ulrich with respect to $Q$ if and only if $N$ is Ulrich with respect to $Q$.
\end{corollary}
\begin{proof}
Since $R$ is Gorenstein and $Q$ is generated by an $R$-sequence, $Q$ must be a Gorenstein ideal (see, e.g., \cite[Proposition 3.1.19(b)]{CMr}). 
Now the assertion is clear by Theorem
\ref{teo3.ulrich.tensor.products} with $I = Q$.\qed
\end{proof}

\begin{corollary}\label{cor3.ulrich.tensor.products3}
Suppose $R$ is Gorenstein. Let $M,N$ be
maximal Cohen-Macaulay $R$-modules such that
$\textrm{\emph{Ext}}_{R}^i(\textrm{\emph{Tr}} M, N)=0$ for all
$i=1, \ldots, d+1$. If $M$ or $N$ is Ulrich, then $M\otimes_R N$ is Ulrich. 
\end{corollary}
\begin{proof}
Apply Theorem \ref{teo3.ulrich.tensor.products} in the classical situation where $I =
{\mathfrak m }$. \qed
\end{proof}

\subsection{The case $N=\omega_R$}
In this part we investigate the Ulrichness of  $M \otimes_R \omega_R$. If $M$ is a finite $R$-module, then as usual we denote by $\textrm{{G-dim}}_RM$ its Gorenstein dimension (see \cite{AuB} for the theory). We start with the following result.

\begin{theorem}\label{teo2.ulrich.tensor.products}
Let $M$ be a finite $R$-module such
that $\textrm{\emph{G-dim}}_RM < \infty$.
\begin{itemize}
\item[(i)] Suppose $I$ is Gorenstein and
$M^*/I M^*$ is $R/I$-free. If $M$ is Ulrich  with
respect to $I$, then $M \otimes_R \omega_R$ is Ulrich with respect to $I$.
             \item[(ii)] Suppose $M/I M$ is $R/I$-free. If
$M \otimes_R \omega_R$ is Ulrich with respect to
$I$, then $M$ is Ulrich with respect to $I$.
\end{itemize}
\end{theorem}

\begin{proof}
(i) Since $\textrm{G-dim}_RM < \infty$ and, in particular, $M$ is maximal
Cohen-Macaulay, then $\textrm{G-dim}_RM = 0$ by the Auslander-Bridger formula. Hence $\textrm{G-dim}_R\textrm{Tr}
M = 0$ (see \cite[Section 2, p.\,5793]{masek}), which by the same token forces $\textrm{depth}_R \textrm{Tr} M =
\textrm{depth}\,R=d$. Therefore,
$$\textrm{Ext}_{R}^i(\textrm{Tr} M, \omega_R)=0 \quad \mbox{for \,all} \quad i \geq 1.$$ On the other hand, \cite[Proposition 12.5]{LeuschkeWiegand2012} furnishes an exact sequence
$$0 \longrightarrow \textrm{Ext}_{R}^1(\textrm{Tr} M, \omega_R)
\longrightarrow M \otimes_R \omega_R \longrightarrow (M^*)^{\dagger} \longrightarrow \textrm{Ext}_{R}^2(\textrm{Tr} M, \omega_R)
\longrightarrow 0.$$ Thus $M \otimes_R \omega_R \cong
(M^*)^{\dagger}$. Now, since $\textrm{Ext}_{R}^i(M,R)=0$ for all $i\geq 1$ (because $\textrm{G-dim}_RM = 0$, which means $M$ is totally reflexive), Lemma \ref{main-lemma}(a) yields that $M^*$ is Ulrich with respect to $I$. So the module $M \otimes_R \omega_R \cong
(M^*)^{\dagger}$ is Ulrich with respect to $I$
by Lemma \ref{ulr}.

\medskip

(ii) Using \cite[Theorem 1]{foxby75} we obtain that the condition
$\textrm{G-dim}_RM < \infty$ implies 
$$M
\cong \textrm{Hom}_{R}(\omega_R, M \otimes_R \omega_R) \quad \mbox{and} \quad 
\textrm{Ext}_{R}^i(\omega_R, M \otimes_R \omega_R)=0 \quad \mbox{for \,all} \quad i\geq 1.$$
To conclude, we are in a position to apply Lemma \ref{main-lemma}(a) to obtain
that the $R$-module $M \cong \textrm{Hom}_{R}(\omega_R, M \otimes_R \omega_R)$ is
Ulrich with respect to $I$. \qed
\end{proof}

\medskip

For the consequence below (an Ulrich-based Gorensteiness criterion), recall that a local ring $R$ is said to be {\it generically Gorenstein} if,  for each minimal prime ${\mathfrak{p}}$ of $R$, the localization $R_{\mathfrak{p}}$ is Gorenstein; this holds if, e.g., $R$ is reduced. Also recall that, for a Cohen-Macaulay local ring $R$ with dimension $d$ and residue field $L$, the (Cohen-Macaulay) type of $R$ is defined as the vector space dimension ${\rm t}
(R)={\rm dim}_L{\rm Ext}_R^d(L, R)$. It is well-known that if $R$ admits a canonical module $\omega_R$, then ${\rm t}(R)=\nu(\omega_R)$.

\begin{corollary}\label{cor2.ulrich.tensor.products2} Suppose $R$ is generically Gorenstein. If there exists an Ulrich $R$-module $M$ with rank, such that $\mbox{\rm G-dim}_RM < \infty$, then $R$ is Gorenstein.
\end{corollary}
\begin{proof} First, we can assume that $R$ possesses a canonical module $\omega_R$ (we pass to the ${\mathfrak m}$-adic completion of $R$ if necessary).
Now, applying Theorem \ref{teo2.ulrich.tensor.products}(i) with $I ={\mathfrak m }$, we obtain that $M \otimes_R \omega_R$ is Ulrich. Then \begin{equation}\label{equals}{\rm e}(M\otimes_R\omega_R)=\nu(M\otimes_R\omega_R)=\nu(M)\nu(\omega_R)={\rm rk}_RM\, {\rm e}(R){\rm t}(R).\end{equation} On the other hand, ${\rm rk}_R\omega_R$ exists and is equal to 1 because $R$ is generically Gorenstein (see \cite[Proposition 3.3.18(a)]{CMr}). Thus, ${\rm rk}_R(M\otimes_R\omega_R)={\rm rk}_RM$ and consequently ${\rm e}(M\otimes_R\omega_R)={\rm rk}_RM\,{\rm e}(R)$. Comparing with (\ref{equals}), it follows that ${\rm t}(R)=1$, i.e., $R$ is Gorenstein.\qed   
\end{proof}

\begin{corollary}\label{cor2.ulrich.tensor.products2_domain} Suppose $R$ is a domain. If there exists an Ulrich $R$-module $M$ with $\mbox{\rm G-dim}_RM < \infty$, then $R$ is Gorenstein.
\end{corollary}

\begin{example}\rm Here we want to illustrate that the above corollaries
can be used to detect infinite Gorenstein dimension. Consider the ideal
 $$P=(xt-ys, xu-zs, yu-zt) \subset S=K[x, y, z, s, t, u]_{(x, y, z, s, t, u)}$$ where $x, y, z, s, t, u$ are 6 indeterminates over a field $K$ of characteristic zero. The  ring $R=S/P$ is a  Cohen-Macaulay local domain of dimension 4 and multiplicity 3. Now let ${\rm Der}_K(R)$ be the $R$-module formed by the $K$-derivations of $R$, i.e., the elements of ${\rm End}_KR$ that satisfy Leibniz rule. Note ${\rm rk}_R{\rm Der}_K(R)=4$. Since $P$ is the ideal generated by the maximal minors of a $2\times 3$ generic matrix, we can apply \cite[Theorem 15.7]{BV} to guarantee that ${\rm Der}_K(R)$ is maximal Cohen-Macaulay. With the aid of \cite{Mac}, we can compute $\nu({\rm Der}_K(R))=12$.
It follows that ${\rm Der}_K(R)$ is Ulrich. Since $R$ is non-Gorenstein, Corollary \ref{cor2.ulrich.tensor.products2_domain} gives $$\mbox{\rm G-dim}_R{\rm Der}_K(R) = \infty.$$ In particular, ${\rm Der}_K(R)$ has infinite projective dimension over $R$.
\end{example}

\section{Applications}\label{appli}

In this last section we provide several applications of some of our results. Especially, we will confirm two famous homological conjectures for the class of Ulrich modules, but first we investigate other topics such as multiplicity of tensor products and freeness criteria for modules, a characterization of complete intersection ideals of codimension three, and an Ulrich-based approach to the long-standing Berger's conjecture about differential modules.

\subsection{Multiplicity and freeness criteria}\label{appli1}
Finding formulas for the multiplicity of a finite module $M$ over a local ring $R$ is a classical issue in commutative algebra and related fields. As recalled in Section \ref{conv+basics}, if  $M$ has a rank then it is easy to see that $\textrm{e}(M)={\rm rk}_RM\, {\rm e}(R)$, but formulas for ${\rm e}(M)$ for general $M$ are desirable as well. Our first observations treat the case of tensor products of modules that may not possess a rank. Curiously we will also see that, in the case where both modules
have a rank, our results give in fact freeness criteria.

\begin{corollary}\label{app-free} Let  $R$ be a local ring of dimension $d\geq 1$, and let $M, N$ be finite $R$-modules. Suppose any one of the following sets of conditions:
    \begin{itemize}
\item[(i)] $R$ is Cohen-Macaulay possessing a canonical module, $M$ is maximal Cohen-Macaulay, $N$ is Ulrich, and  $\textrm{\emph{Ext}}_{R}^i(M,
N^{\dag})=0$ for all $i=1, \ldots, d$;
    \item[(ii)] $R$ is Gorenstein, $M$ and $N$ are maximal Cohen-Macaulay, $N$ {\rm (}resp.\,$M${\rm )} is Ulrich, and  $\textrm{\emph{Ext}}_{R}^i({\rm Tr}\,M, N)=0$ for all $i=1, \ldots, d+1$.
\end{itemize} Then ${\rm e}(M\otimes_RN)=\nu(M)\nu(N)$. If in addition $M$ and $N$ have a rank, and if in situation {\rm (i)} the $R$-module $M$ is reflexive,  then $M$  {\rm (}resp.\,$N${\rm )} is  free.
\end{corollary}
\begin{proof} Assuming (i) (resp.\,(ii)) and applying Theorem \ref{teo.ulrich.tensor.products} (resp.\,Theorem \ref{teo3.ulrich.tensor.products} and Lemma \ref{ulr}) with $I={\mathfrak m}$, we obtain that $M\otimes_RN$ is Ulrich. Hence, ${\rm e}(M\otimes_RN)=\nu(M\otimes_RN)=\nu(M)\nu(N)$, as asserted. Now, suppose $M$ and $N$ have a rank, and for simplicity write ${\rm rk}_RM=r$, ${\rm rk}_RN=s$, and ${\rm e}(R)=e$. In this case, $M\otimes_RN$ has a rank as well, equal to $rs$. Thus, 
${\rm e}(M\otimes_RN)=rse$, which gives \begin{equation}\label{equal}\nu(M)\nu(N)=rse.\end{equation}
Note that if $N$ is Ulrich then  $\nu(N)=se$, which yields $\nu(M)se=rse$ and hence $\nu(M)=r$, so that there is a short exact sequence $0\rightarrow Z\rightarrow R^{r} \rightarrow
M \rightarrow 0$. Notice that ${\rm rk}_RZ=0$, or equivalently, $Z^*=0$. By dualizing this exact sequence,  we get $M^*\cong R^r$, and clearly if (i) holds and $M$ is reflexive then $M$ is necessarily free. Now assume (ii) holds with $N$ Ulrich. The module $M$ must be reflexive because it is maximal Cohen-Macaulay and $R$ is Gorenstein, and then again we get that $M$ is free. Finally, if $M$ is Ulrich in (ii), then by (\ref {equal}) we obtain $\nu(N)=s$. Using the same argument as above, we conclude that $N$ is free. \qed
\end{proof}

\subsection{Complete intersections}\label{appli2} In this subsection we show that Urich modules can be employed to characterize codimension 3 complete intersection ideals. Recall that a proper ideal $J$ of a regular local ring $S$ is said to be {\it generically a complete intersection} if,  for each minimal prime ${\mathfrak{q}}$ of $S$, the ideal $J_{\mathfrak{q}}$ is generated by an  $S_{\mathfrak{q}}$-sequence; note this holds if, e.g., $J$ is radical. The result is as follows.

\begin{corollary}\label{codim3}  Let $R=S/J$, where $(S, {\mathfrak n})$ is a regular local ring of dimension $\delta \geq 4$, and $J$ is a Gorenstein $S$-ideal of height $3$ which is generically a complete intersection. If there exists an Ulrich $R$-module $N$ having a rank and satisfying  $$\textrm{\emph{Ext}}_{R}^i(J/ J^2, N)=0 \quad \mbox{for\, all} \quad i=1, \ldots, \delta -2,$$ then $J=(f, g, h)$ for some $S$-sequence $\{f, g, h\}\subset \mathfrak{n}$.
\end{corollary}
\begin{proof} In this case,  the conormal module $J/J^2$ is maximal Cohen-Macaulay (see \cite{Hu-U}) and has a rank. Moreover, by the Buchsbaum-Eisenbud structure theorem (see \cite{B-E}), the ideal $J$ admits an $S$-free presentation of the form $S^{\nu}\rightarrow S^{\nu}\rightarrow
J \rightarrow 0$ defined by an alternating map $\mathfrak{A} \in {\rm End}_SS^{\nu}$. Then, the conormal module $J/J^2=J\otimes_SR$ has an $R$-free presentation $$R^{\nu}\stackrel{\overline{\mathfrak{A}}}{\longrightarrow} R^{\nu} \longrightarrow
J/J^2 \longrightarrow 0, \quad \overline{\mathfrak{A}}=\mathfrak{A}\otimes {\rm Id}_R.$$ Now set ${\rm Hom}_R(\overline{\mathfrak{A}}, R)=\overline{\mathfrak{A}}^*$. Since $\mathfrak{A}$ is alternating, ${\rm im}\, \overline{\mathfrak{A}}^*= {\rm im}\, \overline{\mathfrak{A}}$ and therefore we can write ${\rm Tr}\,J/J^2={\rm coker}\, \overline{\mathfrak{\mathfrak{A}}}^*= {\rm coker}\, \overline{\mathfrak{A}}=J/J^2$. Also notice that ${\rm dim}\,R=\delta -3\geq 1$. By Corollary \ref{app-free}(ii), the $R$-module $J/J^2$ must be free, which as is well-known is equivalent to $J$ being generated by a regular sequence (see \cite{V}). \qed \end{proof}

\medskip

Concerning this corollary, we may wonder whether, more generally, the vanishing of  $\textrm{{Ext}}_{R}^i(J/ J^2, N)$ in the range $i=t, \ldots, \delta -2$, where $t$ is an integer with $1\leq t\leq \delta -2$, suffices to guarantee that ${\rm pd}_RJ/J^2<t$, where ${\rm pd}_R$ denotes projective dimension over $R$. However, it is now known that the condition ${\rm pd}_RJ/J^2<\infty$
is equivalent to $J$ being generated by an $S$-sequence (see \cite[Theorem A]{B}). So this issue is the same as the following question.

\begin{question}\rm In the setting of Corollary \ref{codim3}, is it true that the single vanishing condition  
$$\textrm{{Ext}}_{R}^{\delta -2}(J/ J^2, N)=0$$ suffices for $J$ to be generated by an $S$-sequence?
\end{question}

\subsection{Berger's conjecture}\label{appli3} In this part, we fix a field $K$ of characteristic zero and consider local $K$-algebras $R$ admitting a universally finite differential module $\Omega_{R/K}$ (for the theory, see \cite{K}, also \cite{Herzog}). This is the case if, for instance, $(R, \mathfrak{m})$ is either 
\begin{equation}\label{two-classes}
{K}[\underline{x}]_{\mathfrak p}/I \ \ \ ({\mathfrak p}\in {\rm Spec}\,{K}[\underline{x}])\quad \mbox{or} \quad {K}[\![\underline{x}]\!]/I,\end{equation} where $I$ is a proper ideal and $\underline{x}=x_1, \ldots, x_m$ is a list of $m$ indeterminates over ${K}$. Recall $\Omega_{R/K}$ is a finite $R$-module, which in the first situation is just the module of K\"ahler differentials of $R$ over ${K}$.



Now let us denote by $\theta \in {\rm Hom}_R(R^m, R^{\nu})$ the map defined by the Jacobian matrix of some (any) generating set of $I$, with entries taken modulo $I$. It is well-known that ${\rm Tr}\,\Omega_{R/K}={\rm Coker}\,\theta$. We consider the long-standing {\it Berger's conjecture}, raised over 60 years ago in \cite{Ber}; see also the surveys \cite{BeSur}, \cite{Herzog}.

\bigskip

\noindent {\bf Conjecture (Berger)}\label{Berger}
Let $(R, \mathfrak{m})$ be as in $($\ref{two-classes}$)$. Suppose $R$ is reduced and one-dimensional. If $\Omega_{R/K}$ is torsionfree, then $R$ is regular.

 \bigskip

This problem remains open even if $R$ is Gorenstein. Our result is as follows.

\begin{corollary}\label{app-Berger} Berger's conjecture is true in any one of the following situations:
\begin{itemize}
\item[(i)] There exists an Ulrich $R$-module $N$ with rank, such that $${\rm Ext}_R^1(\Omega_{R/K}, N^{\dag})=0;$$

\item[(ii)] $R$ is Gorenstein and there exists an Ulrich $R$-module $N$ with rank, such that $${\rm Ext}_R^1({\rm Coker}\,\theta, N)={\rm Ext}_R^2({\rm Coker}\,\theta, N)=0;$$

\item[(iii)] $R$ is Gorenstein and there exists an Ulrich $R$-module $M$ with rank, such that $${\rm Ext}_R^1({\rm Tr}\,M, \Omega_{R/K})={\rm Ext}_R^2({\rm Tr}\,M, \Omega_{R/K})=0.$$

\end{itemize}
\end{corollary}
\begin{proof} First note that, under the hypotheses of the conjecture, $R$ is Cohen-Macaulay and $\Omega_{R/K}$ is a maximal Cohen-Macaulay $R$-module with rank (equal to ${\rm dim}\,R$). Now, in any of the three cases, the proof of Corollary \ref{app-free} shows that $\nu(\Omega_{R/K})={\rm rk}_R\Omega_{R/K}$ and consequently $\Omega_{R/K}$ must be cyclic. On the other hand, in the situation $R=K [\underline{x}]_{\mathfrak p}/I$ (resp. $R={K}[\![\underline{x}]\!]/I$), we can use \cite[Corollary 6.5(b)]{K} (resp. \cite[Corollary 13.15]{K}) to obtain $$\nu(\Omega_{R/K})=\nu(\mathfrak{m}).$$ It follows that $\mathfrak{m}$ is a principal ideal, i.e., $R$ is regular. \qed \end{proof}

\subsection{Regularity, and two homological conjectures}\label{appli4} Our goal in this final part is to confirm, for the class of Ulrich modules, two famous conjectures concerning the effects of Ext vanishing. 
We begin providing a central result  which is, in essence, a byproduct of Corollary \ref{cor4.ulrich.tensor.products} and additionally gives an Ulrich-based characterization of regular local rings. 


\begin{corollary}\label{cor5.ulrich.tensor.products}
Let $R$ be a Cohen-Macaulay local domain of dimension $d\geq 1$. Let $M$ be an Ulrich $R$-module. The following conditions are equivalent:
\begin{itemize}
\item[(i)] ${\rm Ext}_{R}^i(M, M)=0$ for all $i\geq 1$;

\item[(ii)] ${\rm Ext}_{R}^i(M, M)=0$ for all $i=1, \ldots, d$;
             \item[(iii)] $R$ is regular and $M$ is free;
             \item[(iv)]  $R$ is regular.
\end{itemize}
\end{corollary}
\begin{proof} First, by passing to the ${\mathfrak m}$-adic completion of $R$ if necessary, we can assume that $R$ possesses a canonical module $\omega_R$. The implication (i)$\Rightarrow$(ii) is clear. Now suppose (ii). According to Corollary \ref{cor4.ulrich.tensor.products}, the $R$-module $M \otimes_R M^{\dagger}$
is Ulrich. Therefore,  $\textrm{e}(M \otimes_R M^{\dagger})=\nu (M \otimes_R M^{\dagger})=\nu(M)\nu(M^{\dagger})$. Since $R$ is a domain, every finite $R$-module has a rank. As ${\rm rk}_R\omega_R=1$, it is easy to see that ${\rm rk}_RM={\rm rk}_RM^{\dagger}$, and let us denote this common number by $r$.  Clearly, $\textrm{rk}_R(M \otimes_R M^{\dagger})=r^2$, and write ${\rm e}(R)=e$. Thus,
 \begin{equation}\label{multi-tensor}\textrm{e}(M \otimes_R M^{\dagger}) =r^2e. \end{equation} On the other hand, by Lemma \ref{ulr}, the $R$-module $M^{\dagger}$ must be Ulrich because so is $M$. It follows that $\nu(M)$ and   $\nu(M^{\dagger})$ are both equal to $re$, and hence $\textrm{e}(M \otimes_R M^{\dagger})=r^2e^2$. Comparing with (\ref{multi-tensor}), we get $e=1$, and since the local ring $R$ is unmixed (being Cohen-Macaulay), it must be regular by \cite[Theorem 40.6]{Na}. Now, we obtain $\nu(M)=re=r$, which is equivalent (since $R$ is a local domain) to $M$ being free. This gives (iii), which obviously implies (iv). Finally, if (iv) holds, then $e=1$ and as before $\nu(M)=r$, i.e., $M$ is free, which implies (i). \qed
\end{proof}

\medskip

A finite $R$-module $C$ is said to be {\it semidualizing} if the map $R \rightarrow {\rm Hom}_R(C , C)$ given by homothety is an isomorphism and if, in addition, ${\rm Ext}^i_R(C , C ) = 0$ for all $i \geq 1$. For example, if $R$ is a Cohen-Macaulay local ring having a canonical module $\omega_R$, then $\omega_R$ is semidualizing. Also recall that if $R$ is a Gorenstein (e.g., regular) local ring then necessarily $C\cong R$. We just want to record the following consequence, which retrieves, in the case of domains of positive dimension, a result from \cite{MirandaQueirozSouza}. 

\begin{corollary}{\rm (\cite[Corollary 3.9]{MirandaQueirozSouza})} Let $R$ be a Cohen-Macaulay local domain of positive dimension. Let $C$ be a semidualizing $R$-module. Then, $R$ is regular if and only if $C$ is Ulrich.
\end{corollary}


The first conjecture we invoke in this subsection is the following well-known problem in dimension one. It was originally proposed in \cite{HW} and later reformulated in \cite{HIW}.

\bigskip

\noindent {\bf Conjecture (Huneke-Wiegand)}\label{hiwconjecture}
Let $R$ be a one-dimensional Gorenstein local domain. If $M$ is a torsionfree finite $R$-module such that ${\rm Ext}^1_R(M, M)=0$, then $M$ is free.

 \bigskip

Note the clear connection between the Huneke-Wiegand conjecture and Berger's conjecture: if the former is true, then the latter also holds true (in case $R$ is a Gorenstein domain) if, in addition, ${\rm Ext}^1_R(\Omega_{R/K}, \Omega_{R/K})=0$. This fact is aligned, in nature, with Corollary \ref{app-Berger}.

Our contribution to the Huneke-Wiegand conjecture, which we now state, follows immediately from the case $d=1$ of Corollary \ref{cor5.ulrich.tensor.products}. 

\begin{corollary}\label{HW-Ulrich} The Huneke-Wiegand conjecture is true if
$M$ is an Ulrich $R$-module.
\end{corollary}

\begin{example}\rm The above corollary is no longer true if we relax the condition of $R$ being a domain to that of being reduced, even if $R$ is an isolated complete intersection singularity. The simplest instance is when $R=K[\![x, y]\!]/(xy)$ (here, $x, y$ are formal indeterminates over a field $K$) and $M=R/xR$. Notice that ${\rm e}(R)=2$. By the short exact sequence 
$$0\longrightarrow xR\longrightarrow R \longrightarrow
yR \longrightarrow 0$$ it is easy to see that
 the $R$-module $M\cong yR$ is Ulrich and satisfies $${\rm Ext}^1_R(M, M)\cong 0:_M(0:_Rx)/xM\cong 0:_My=0,$$ while clearly $M$ cannot be free (it has a non-trivial annihilator).
\end{example}

\smallskip

To close the paper, we consider one of the most celebrated problems in homological commutative algebra, to wit, the following commutative local version of the conjecture from \cite[p.\,70]{AR}, which has been extensively explored in the literature but remains open even if $R$ is Gorenstein (among a number of other situations). We refer to \cite{HL}, \cite{Ku}, and their lists of references on the subject.

\bigskip

\noindent {\bf Conjecture (Auslander-Reiten)}\label{arconjecture}
Let $R$ be a local ring. If $M$ is a finite $R$-module such that ${\rm Ext}^i_R(M, R)={\rm Ext}^i_R(M, M)=0$ for all $i\geq 1$, then $M$ is free.
 
 \bigskip

Regarding this conjecture, we have  the following immediate byproduct of Corollary \ref{cor5.ulrich.tensor.products} (note we can assume $d\geq 1$, since a zero-dimensional integral domain is simply a field).

\begin{corollary}\label{AR-Ulrich} The Auslander-Reiten conjecture is true if $R$ is a Cohen-Macaulay local domain and $M$ is an Ulrich $R$-module.
\end{corollary}

\bigskip

\noindent{\bf Acknowledgements.} The first author was partially supported by the CNPq-Brazil grants 301029/2019-9 and 406377/2021-9. The second author was partially supported by the FAPESQ grant 3099/2021.

\end{document}